\documentclass[preprint,11pt]{elsarticle}
\usepackage[utf8]{inputenc}

\journal{DGA}

\usepackage{hyperref}

\usepackage{ifpdf}

\usepackage[dvipsnames]{xcolor}

\usepackage{amsmath,amsthm,amssymb}
\usepackage{booktabs}
\usepackage[bf]{caption}

\usepackage{eucal}

\usepackage{relsize}

\usepackage{boxedminipage}

\usepackage{listings}

\usepackage[all]{xy}

\newtheorem{theorem}{Theorem}[section]
\newtheorem{lemma}{Lemma}[section]
\newtheorem{proposition}{Proposition}[section]
\newtheorem{corollary}{Corollary}[section]

\theoremstyle{definition}

\newtheorem{remark}{Remark}[section]
\newtheorem{example}{Example}[section]
\newtheorem{counterexample}[example]{Counter-example}

\DeclareMathOperator{\real}{Re}

\DeclareMathOperator{\grad}{grad}
\DeclareMathOperator{\sgrad}{sgrad}
\DeclareMathOperator{\divv}{div}
\DeclareMathOperator{\curl}{curl}

\newcommand{\e}{\mathrm{e}}
\renewcommand{\i}{\mathrm{i}}
\newcommand{\pair}[1]{\langle #1 \rangle}
\newcommand{\inner}[1]{\langle\!\langle #1 \rangle\!\rangle}

\providecommand{\abs}[1]{\lvert#1\rvert}

\newcommand{\trans}{\top}
\DeclareMathOperator{\image}{im}

\newcommand{\ud}{\,\mathrm{d}}
\newcommand{\dd}{\mathrm{d}}
\newcommand{\pd}{\partial}

\newcommand{\R}{{\mathbb R}}
\newcommand{\C}{{\mathbb C}}

\newcommand{\disk}{\mathbb{D}}

\newcommand{\Fcal}{\mathcal{F}}
\newcommand{\LieD}{\pounds}
\newcommand{\interior}{\mathrm{i}}

\newcommand{\Diff}{\mathrm{Diff}}
\newcommand{\Emb}{\mathrm{Emb}}
\newcommand{\Xcal}{\mathfrak{X}}

\newcommand{\vol}{\mathrm{vol}}

\newcommand{\Xcalvol}{{\Xcal_\vol}}

\newcommand{\con}{\mathrm{con}}

\newcommand{\Xcalcon}{{\Xcal_\con}}
\newcommand{\Con}{\mathrm{Con}}

\newcommand{\iso}{\mathrm{iso}}

\newcommand{\hol}{\mathrm{hol}}

\newcommand{\Acal}{\mathcal{A}}
\newcommand{\U}{\mathsf{U}}
\newcommand{\Id}{\mathsf{Id}}

\newcommand{\OLD}[1]{}

\begin{document}

\begin{frontmatter}


\title{On Hodge decomposition and conformal variational problems\tnoteref{titlelabel}}
\tnotetext[titlelabel]{This research was supported by the Marsden Fund.}

\author[seat]{Stephen Marsland}
\author[ifs]{Robert McLachlan}
\author[ifs]{Klas Modin\corref{cor1}\fnref{modinlabel}}
\ead{K.E.F.Modin@massey.ac.nz}
\fntext[modinlabel]{Supported in part by the Royal
Physiographic Society in Lund, Hellmuth Hertz' foundation grant.}
\cortext[cor1]{Corresponding author.}
\author[ifs]{Matthew Perlmutter}


\address[seat]{School of Engineering and Advanced Technology,
Massey University \\
Private Bag 11 222, Palmerston North 4442, New Zealand}
\address[ifs]{Institute of Fundamental Sciences,
Massey University \\
Private Bag 11 222, Palmerston North 4442, New Zealand}





\begin{abstract}
	
	The main result is the identification of the orthogonal complement
	of the subalgebra of conformal vector field inside the algebra of 
	all vector fields of a compact flat 2--manifold.
	As a fundamental tool, the complete Hodge
	decomposition for manifold with boundary is used.
	The identification allows the derivation of
	governing differential equations for variational
	problems on the space of conformal vector fields.
	Several examples are given. 
	In addition, the paper also gives a
	review, in full detail, of already known vector field decompositions
	involving subalgebras of volume preserving and symplectic vector fields.
\end{abstract}


\end{frontmatter}

%
%
%
%
%
%


\section{Introduction}

It is well known that the Euler equations of fluid dynamics correspond
to a Lie--Poisson equation on the infinite dimensional Lie algebra of
volume preserving vector fields~\cite{Ar1966,ArKh1998}.
On domains of~$\R^2$ or~$\R^3$ the derivation of the strong form of the Euler equations
starts from a variational principle, and relies on the Helmholtz
decomposition in order to identify the $L^2$~orthogonal complement
of the space of volume preserving vector fields inside the space
of all vector fields. A generalisation to arbitrary compact Riemannian manifolds,
with or without boundary, is obtained by identifying vector fields
with 1--forms (by contraction with the metric), and then use
the Hodge decomposition for manifolds with boundary in order to identifying the
$L^2$~orthogonal complement~\cite{EbMa1970}.

In this paper we give a framework for how various Lie subalgebras
of vector fields can be identified with one or several components in the
Hodge decomposition by suitable isometric isomorphisms between vector fields
and 1--forms, thus identifying the $L^2$~orthogonal complement of the subalgebra. 
The first example concerns Lie subalgebras
of volume preserving vector fields. This example is well known
in the literature, although it is hard to find a detailed exposition
for the cases that occur for a manifold with boundary. Thus,
the first aim of the paper is to give such an exposition.
The second example concerns Lie subalgebras of symplectic and Hamiltonian vector fields.
It is known that the space of symplectic vector fields can by identified
with closed 1--forms (by contraction with the symplectic form).
The second aim of the paper is to give a detailed exposition
of various $L^2$~orthogonal decompositions involving symplectic and Hamiltonian vector fields,
on a compact almost Kähler
manifold with boundary. 
%
%
%
%
The third example is new and constitutes the main contribution of the paper.
We show how the subalgebra of conformal vector fields on a flat 2--manifold can be
identified with the space of harmonic fields. As an application, we then give examples
of how to derive partial differential equations from variational problems on the space
of conformal vector fields. In the main example, which is also the original motivation for this
paper, we derive a geodesic equation on the infinite dimensional manifold of planar conformal
embeddings of the unit disk.

The paper is organised as follows. In Section~\ref{sec:lie_algebras_of_vector_fields}
we give a detailed review of the three types of vector field subalgebras.
We work in the category of Fréchet-Lie algebras, and we give proofs that
the subalgebras considered are proper Fréchet-Lie subalgebras. The results in this
section does not require the underlying manifold to be compact.

In Section~\ref{sec:review_of_the_hodge_decomposition} we first review
the Hodge decomposition for manifolds with boundary. 
Thereafter, we use the standard Hodge decomposition for manifolds with
boundary in combination with the so called Friedrich decomposition to obtain
a complete Hodge decomposition, involving six spaces,
and we show that these spaces can be characterised in terms of kernels
and images of the differential and co-differential.
The special case of 2--manifolds is studied in further detail, and
the unit disk, the standard annulus, the torus, and the sphere
are given as examples.

In Section~\ref{sec:orthogonal_decomposition_of_vector_fields} we
use the Hodge decomposition to obtain various $L^2$~orthogonal decompositions
of vector fields. Altogether, we derive nine different decompositions,
which are summarised in Table~\ref{tab:decompositions}.
The main result is Theorem~\ref{thm:conformal_decomposition},
which gives a decomposition of vector fields on a flat 2--manifold,
involving conformal vector fields as one of the components in the Hodge decomposition.
The section ends with various examples and one counter-example.

In Section~\ref{sec:application_examples} we work out the governing
differential equations for three variational problems on the space
of conformal vector fields on a simply connected bounded domain of~$\R^2$.

%

\section{Lie algebras of vector fields} 
\label{sec:lie_algebras_of_vector_fields}

In this section, let~$M$ be an $n$-manifold, possibly with
boundary, such that~$M$ is either compact, or can be equipped with a
countable sequence of compact sets $K_{i}\subset M$ such that each compact
subset $U\subset M$ is contained in one of~$K_{i}$.
In this case, the linear space $\Fcal(M)$~of smooth real valued function 
on~$M$ can be equipped with a sequence of semi-norms, making it
a Fréchet space (see e.g.~\cite{Ha1982} for details). In turn,
this induces a Fréchet topology on the space $\mathcal{T}_{r}^{s}(M)$
of smooth tensor fields on~$M$ of finite order.
In particular, the space~$\Xcal(M)$ of vector fields, and
the space $\Omega^{k}(M)$ of $k$--forms are Fréchet spaces.
Furthermore, the topologies
are ``compatible'' with each other, in the sense that
any partial differential operator with smooth coefficients
between any two spaces of tensor fields
is a smooth map~\cite[Sect.~II.2.2]{Ha1982}.
%
%
In particular, the Lie derivative map
\begin{equation*}
	\Xcal(M)\times\mathcal{T}_{r}^{s}(M)\ni (\xi,t)\mapsto \LieD_{\xi}t \in\mathcal{T}_{r}^{s}(M)
\end{equation*}
is smooth, which in turn implies that~$\Xcal(M)$ is a Fréchet-Lie algebra
with Lie bracket given by $[\xi,\eta] = - \LieD_{\xi}\eta$
(this bracket fulfils the Jacobi identity).

Recall that a subspace of a Fréchet-Lie algebra is called a Fréchet-Lie subalgebra
if it is topologically closed, and also closed under the Lie bracket.
In the case when $M$ has a boundary, 
it holds that the subspace $\Xcal_{\mathsf{t}}(M)$ of vector fields that are tangential
to the boundary is a Fréchet-Lie subalgebra.
Basically, this is the only subalgebra which can be obtained
intrinsically, without introducing any further structures on~$M$.
In the remainder of this section we review some other well known subalgebras of vector fields
which require extra structure on the manifold.

\subsection{Volume preserving vector fields} 
\label{sub:volume_preserving_vector_fields}

Assume that $M$ is orientable, and let $M$ be
equipped with a volume form, denoted~$\vol$. The set of
volume preserving vector fields is then given by
$\Xcalvol(M) = \{ \xi\in\Xcal(M); \LieD_{\xi}\vol = 0 \}$.
It is clear that this is a linear subspace of $\Xcal(M)$.
Recall that the divergence with respect to~$\vol$ is the partial differential
operator $\divv:\Xcal(M)\to\Fcal(M)$ defined by
$\LieD_{\xi}\vol = \divv(\xi)\vol$ for all $\xi\in\Xcal(M)$.
Thus, since the volume form is strictly non-zero, it holds that
$\xi\in\Xcalvol(M)$ if and only if $\divv(\xi) = 0$.
\begin{proposition}\label{pro:volume_algebra}
	$\Xcalvol(M)$ is a Fréchet-Lie subalgebra of $\Xcal(M)$.
\end{proposition}

\begin{proof}
	The differential operator $\divv:\Xcal(M)\to\Fcal(M)$ is smooth
	in the Fréchet topology. In particular, it is continuous, so the preimage
	of the closed set $\{0\}\in\Fcal(M)$, which is equal to $\Xcalvol(M)$,
	is also closed. Thus, $\Xcalvol(M)$ is a topologically closed subspace of~$\Xcal(M)$.
	
	Next, let $\xi,\eta\in\Xcalvol(M)$. Then
	\begin{equation*}
		\LieD_{[\xi,\eta]}\vol = \LieD_{\LieD_{\eta}\xi}\vol
		= \LieD_{\eta}\underbrace{\LieD_{\xi}\vol}_{0} - \LieD_{\xi}\underbrace{\LieD_{\eta}\vol}_{0}
		= 0 .
	\end{equation*}
	Thus, $\Xcalvol(M)$ is closed under the Lie bracket, which finishes the proof.
\end{proof}

In the case when $M$ has a boundary, it also holds that
the subspace $\Xcal_{\vol,\mathsf{t}}(M) := \Xcalvol(M)\cap\Xcal_{\mathsf{t}}(M)$
is a Fréchet-Lie subalgebra. This follows immediately since both $\Xcalvol(M)$
and $\Xcal_\mathsf{t}(M)$ are Fréchet-Lie subalgebras.

Next, consider the subspace of \emph{exact} divergence free vector fields given by
\begin{equation*}
	\Xcal^\mathrm{ex}_\vol(M) = \{ \xi\in\Xcal(M); \interior_\xi\vol\in\dd\Omega^{n-2}(M) \}.
\end{equation*}
The following result is well known (see e.g.~\cite{ArKh1998}).

\begin{proposition}\label{pro:exact_volume_algebra}
	$\Xcal_\vol^\mathrm{ex}(M)$ is a Fréchet-Lie subalgebra of $\Xcal(M)$
	and an ideal in $\Xcalvol(M)$.
\end{proposition}

\begin{proof}
	Topological closeness follows since the maps 
	$ \Omega^{n-1}(M)\ni\interior_\xi\vol \mapsto \xi \in \Xcal(M)$
	and $\dd:\Omega^{n-2}(M)\to\Omega^{n-1}(M)$
	are smooth in the Fréchet topology. Next, if $\xi\in\Xcal_\vol^\mathrm{ex}(M)$
	then $\xi$ is divergence free since 
	$\LieD_\xi\vol = \dd\interior_\xi\vol = 0$. Finally, if $\eta\in\Xcalvol(M)$
	then $\interior_{\LieD_\eta\xi}\vol = 
	\LieD_\eta \interior_\xi\vol = \LieD_\eta\dd\alpha = \dd\LieD_\eta\alpha \in \dd\Omega^{n-2}(M)$,
	so $\Xcal_\vol^\mathrm{ex}(M)$ is an ideal in $\Xcalvol(M)$.
\end{proof}

Continuing as before, we also obtain the smaller Fréchet-Lie subalgebra of tangential
exact divergence free vector fields, by 
\begin{equation*}
	\Xcal_{\vol,\mathsf{t}}^\mathrm{ex}(M) = \Xcal_\vol^\mathrm{ex}(M)\cap\Xcal_\mathsf{t}(M).
\end{equation*}

The space of volume preserving vector fields is of importance in
fluid mechanics. In particular, the motion of an incompressible
ideal fluid is described by a differential equation evolving
on the phase space $\Xcal_{\vol,\mathsf{t}}(M)$,
which is the Lie algebra of the infinite dimensional Lie group of volume preserving diffeomorphisms of~$M$~\cite{Ar1966}.


\subsection{Symplectic vector fields} 
\label{sub:symplectic_vector_fields}

Let $M$ be equipped with a symplectic structure, i.e., a
closed non-degenerate 2--form $\omega$.
%
Then the subspace of symplectic vector fields on~$M$ is given by 
$\Xcal_{\omega}(M) = \{ \xi\in\Xcal(M); \LieD_{\xi}\omega = 0 \}$.

\begin{proposition}\label{pro:symplectic_vf}
	$\Xcal_{\omega}(M)$ is a Fréchet-Lie subalgebra of $\Xcal(M)$.
\end{proposition}

\begin{proof}
	The map $\Xcal(M) \ni \xi \mapsto \LieD_{\xi}\omega \in \Omega^{2}(M)$
	is smooth, so its preimage of $\{0\}\in\Omega^{2}(M)$ is topologically closed.
	Thus, $\Xcal_{\omega}(M)$ is topologically closed in $\Xcal(M)$.
	Further, if $\xi,\eta\in\Xcal_{\omega}(M)$, then
	\begin{equation*}
		\LieD_{[\xi,\eta]}\omega = \LieD_{\eta}\underbrace{\LieD_{\xi}\omega}_{0}-
		\LieD_{\xi}\underbrace{\LieD_{\eta}	\omega}_{0} = 0 .
	\end{equation*}
	Thus, $\Xcal_{\omega}(M)$ is closed under bracket, which concludes the proof.
\end{proof}

The space of Hamiltonian vector fields are those who have
a globally defined Hamiltonian. That is,
\begin{equation*}
	\Xcal_\mathrm{Ham}(M) = \{ \xi\in\Xcal(M); \interior_\xi\omega \in \dd\Omega^0(M) \}.
\end{equation*}
With the same proof as for Proposition~\ref{pro:exact_volume_algebra},
but replacing $\vol$ with $\omega$, we get the following result.

\begin{proposition}\label{pro:exact_symplectic_algebra}
	$\Xcal_\mathrm{Ham}(M)$ is a Fréchet-Lie subalgebra of $\Xcal(M)$
	and an ideal in $\Xcal_\omega(M)$.
\end{proposition}

Just as in the volume preserving case, we also have the
smaller Fréchet-Lie subalgebras of symplectic and Hamiltonian tangential vector fields,
\begin{equation*}
	\Xcal_{\omega,\mathsf{t}}(M) = \Xcal_\omega(M)\cap\Xcal_\mathsf{t}(M)
	\qquad\text{and}\qquad
	\Xcal_{\mathrm{Ham},\mathsf{t}}(M) = \Xcal_\mathrm{Ham}(M)\cap\Xcal_\mathsf{t}(M) .
\end{equation*}



\subsection{Conformal vector fields} 
\label{sub:conformal_vector_fields}


Let $M$ be equipped with a Riemannian metric~$\mathsf{g}$.
%
Then the subspace of conformal vector fields is given by
\[
	\Xcal_\con(M) = \{ \xi\in\Xcal(M); 
	\LieD_{\xi}\mathsf{g} = F\mathsf{g}, F\in\Fcal(M) \}.
\] 
Thus, if $\xi\in\Xcalcon(M)$
then $\xi$ preserves the metric up to scaling by a function. In turn,
this implies that the infinitesimal transformation generated
by $\xi$ preserve angles. 
%
%
%
Indeed, if $\eta,\psi\in\Xcal(M)$
are everywhere orthogonal, i.e., $\interior_{\eta}\interior_{\psi}\mathsf{g} = \mathsf{g}(\eta,\psi)=0$, then
$\mathsf{g}(\LieD_{\xi}\eta,\psi) + \mathsf{g}(\eta,\LieD_{\xi}\psi)=0$, which follows
since 
\[
0 = \LieD_{\xi}(\mathsf{g}(\eta,\psi)) = F\mathsf{g}(\eta,\psi)
+ \mathsf{g}(\LieD_{\xi}\eta,\psi) + \mathsf{g}(\eta,\LieD_{\xi}\psi) =
\mathsf{g}(\LieD_{\xi}\eta,\psi) + \mathsf{g}(\eta,\LieD_{\xi}\psi).
\]

%

\begin{proposition}\label{pro:conformal_vector_fields_algebra}
	$\Xcal_\con(M)$ is a Fréchet-Lie subalgebra of $\Xcal(M)$.
\end{proposition}

\begin{proof}
	We need to show that $\Xcal_\con(M)$ is closed 
	under the Lie bracket 
	and that $\Xcal_\con(M)$ is topologically closed in $\Xcal(M)$.
	Let $\xi,\eta\in\Xcal_\con(M)$. Then
	\[
		\LieD_{\LieD_{\xi}\eta} \mathsf{g} = \LieD_{\xi}\LieD_{\eta}\mathsf{g}-
		\LieD_{\eta}\LieD_{\xi}\mathsf{g}
		= \LieD_{\xi}(G \mathsf{g}) - \LieD_{\eta}(F \mathsf{g})
		= (\LieD_{\xi}G-\LieD_{\eta}F)\mathsf{g},
	\]
	which proves that $[\xi,\eta] = -\LieD_{\xi}\eta\in\Xcal_\con(M)$.
	
	To prove that $\Xcal_\con(M)$ is topologically closed in $\Xcal(M)$,
	we define a map $\Phi:\Xcal(M)\to \mathcal{T}^0_2(M)$ by
	$\xi\mapsto \LieD_{\xi}\mathsf{g}$. This is a smooth
	map in the Fréchet topology.
	We notice that $\Xcal_\con(M)=\Phi^{-1}(\Fcal(M)\mathsf{g})$. 
	Since $\Fcal(M)\mathsf{g}$ is topologically closed in $\mathcal{T}_2^0(M)$
	it follows from continuity of~$\Phi$ that its preimage, i.e., $\Xcal_\con(M)$,
	is topologically closed in~$\Xcal(M)$.
	%
	%
	%
\end{proof}

Notice that the condition $\LieD_{\xi}\mathsf{g} = F\mathsf{g}$ for a vector field~$\xi$
to be conformal is not as
``straightforward'' as the conditions for being volume preserving or symplectic, since
the function~$F$ depends implicitly on~$\xi$. We now work out an explicit coordinate
version of this condition in the case when the manifold~$M$ is conformally flat.

First, recall that a Riemannian manifold is locally conformally flat if for every
element $z\in M$ there exists a neighborhood~$U$ of~$z$ 
and a function~$f\in\Fcal(U)$ such that~$e^{2f}\mathsf{g}$ is a flat metric
on~$U$. Thus, we may chose local coordinates mapping~$U$ conformally to
flat Euclidean space, i.e., such that $\mathsf{g} = c \sum_{i}\dd x^{i}\otimes \dd x^{i}$,
with $c=e^{2 f}$.
Next, consider a vector field expressed in these coordinates
$\xi = \sum_{i}u^{i}\frac{\pd}{\pd x^{i}}$. Then
\begin{equation*}
	\begin{split}
		\LieD_{\xi}\mathsf{g} &= (\LieD_{\xi} c) \sum_{i}\dd x^{i}\otimes\dd x^{i}
		+ c \sum_{i} \big(\dd u^i\otimes\dd x^i + \dd x^i \otimes \dd u^i \big)
		\\
		&= \sum_{i} 
			(2 c \frac{\pd u^i}{\pd x^i} + \interior_{\xi}\dd c)\dd x^i\otimes\dd x^i + 
			\sum_{i< j} c \big( \frac{\pd u^i}{\pd x^j}+\frac{\pd u^j}{\pd x^i}\big) 
			\big(\dd x^i \otimes \dd x^j+ \dd x^j \otimes \dd x^i\big) .
		\\
		&= \sum_{i} 
			2(\frac{\pd u^i}{\pd x^i} + \interior_{\xi}\dd f) c \ud x^i\otimes\dd x^i + 
			\sum_{i< j} c \big( \frac{\pd u^i}{\pd x^j}+\frac{\pd u^j}{\pd x^i}\big) 
			\big(\dd x^i \otimes \dd x^j+ \dd x^j \otimes \dd x^i\big) .
	\end{split}
\end{equation*}
From this we get that $\LieD_{\xi}\mathsf{g}$ is pointwise parallel with
$\mathsf{g}$ if the components of~$\xi$ fulfill the following 
$n(n+2)/2 - 1$ relations
\begin{equation}\label{eq:coordinate_conformal_conditions}
	\left\{
	\begin{split}
		\frac{\pd u^{i}}{\pd x^{i}} - \frac{\pd u^{i+1}}{\pd x^{i+1}} &= 0 \qquad \forall\, i < n \\
		\frac{\pd u^{i}}{\pd x^{j}} + \frac{\pd u^{j}}{\pd x^{i}} &= 0  \qquad \forall\, i < j  .
	\end{split}
	\right.
\end{equation}
Notice that these equations are independent of the function~$c$. Also,
notice that if $M$ is a 2-manifold, these are the Cauchy-Riemann equations.

The coordinate formula (in conformally flat coordinates) for the
divergence with respect to the
volume form induced by~$\mathsf{g}$ is given by
\begin{equation*}
	\divv(\xi)= \interior_{\xi}\dd c + 
	c\sum_{i}\frac{\pd u^{i}}{\pd x^{i}} .
\end{equation*}
Thus, we see that if $M$ is a 2--manifold, and $\xi\in\Xcalcon(M)$,
then locally we have $\LieD_{\xi}\mathsf{g} = \divv(\xi)/c\, \mathsf{g}$.
In particular, if $(M,\mathsf{g})$~is a flat 2--manifold, then 
it holds that $\LieD_{\xi}\mathsf{g} = \divv(\xi)\mathsf{g}$.

\section{Hodge decomposition} 
\label{sec:review_of_the_hodge_decomposition}

In this section, let $(M,\mathsf{g})$ be a compact oriented $n$--dimensional Riemannian manifold
possibly with boundary,
and let $\Omega^k(M)$ denote the space of smooth $k$--forms on~$M$.
We sometimes use the notation $\Fcal(M)$ for $\Omega^0(M)$.
Recall the Hodge star operator $\star:\Omega^k(M)\to\Omega^{n-k}(M)$,
which is defined in terms of the metric (see~\cite[Chap.~6]{AbMaRa1988}).
Using the Hodge star, the space $\Omega^k(M)$ is 
equipped with the $L^2$~inner product:
\[
	\pair{\alpha,\beta}_{M} := \int_M \alpha\wedge\star\beta .
\]
Up to a boundary integral term, the 
\emph{co--differential} $\delta:\Omega^k(M)\to\Omega^{k-1}(M)$
is the formal adjoint of the differential with respect to the $L^2$~inner product.
Indeed, it holds that
\[
	\pair{\ud\gamma,\beta}_{M} = 
	\pair{\gamma,\delta\beta}_{M} + 
	\int_{\pd M}\gamma\wedge\star\beta .
\]
The explicit formula is $\delta = (-1)^{(n-k+1)k}\star\!\dd\star$.


The subspaces $\Omega^k_{\mathsf{t}}(M),\Omega^k_{\mathsf{n}}(M)\subset\Omega^k(M)$
of tangential and normal $k$--forms 
are defined by
\[
	\begin{split}
		\Omega^k_{\mathsf{t}}(M) &= \{ \alpha\in\Omega^k(M); i^*(\star\alpha)=0 \}
		\\
		\Omega^k_{\mathsf{n}}(M) &= \{ \alpha\in\Omega^k(M); i^*(\alpha)=0 \}
	\end{split}
\]
where $i:\pd M \to M$ is the inclusion.
We also have the subspace of $k$--forms that vanish
on the boundary: $\Omega^k_0(M)=\Omega^k_{\mathsf{t}}(M)\cap\Omega^k_{\mathsf{n}}(M)$.
Notice that $\Omega^0_{\mathsf{n}}(M)=\Omega^0_0(M)$, $\Omega^0_{\mathsf{t}}(M)=\Omega^0(M)$,
$\Omega^n_{\mathsf{t}}(M)=\Omega^n_0(M)$, and $\Omega^n_{\mathsf{n}}(M)=\Omega^{n}(M)$.
It also holds that $\star\,\Omega^k_{\mathsf{t}}(M) = \Omega^{n-k}_{\mathsf{n}}(M)$.

The following
fundamental result is known as the
Hodge decomposition theorem for manifolds with boundary (see e.g.~\cite{Sc1995}):

\begin{theorem}\label{thm:Hodge_boundary_general}
	%
	$\Omega^k(M)$ admits the $L^2$~orthogonal
	decomposition
	\[
		\Omega^k(M) = \ud\Omega_{\mathsf{n}}^{k-1}(M)\oplus\delta
			\Omega^{k+1}_{\mathsf{t}}(M)\oplus\mathcal{H}^k(M),
	\]
	where $\mathcal{H}^k(M) = \{ \alpha\in\Omega^k(M); \ud\alpha=0, \delta\alpha=0 \}$
	are the harmonic $k$--fields. 
\end{theorem}

The harmonic fields $\mathcal{H}^{k}(M)$ can be further decomposed
in two different ways, called
Friedrich decompositions (see e.g.~\cite{Sc1995}).

\begin{theorem}\label{thm:Friedrich_decomposition}
	$\mathcal{H}^{k}(M)$ admits the $L^{2}$~orthogonal decompositions
	\begin{equation*}
		\begin{split}
			\mathcal{H}^{k}(M) &= \mathcal{H}_{\mathsf{t}}^{k}(M)\oplus \{ \alpha \in\mathcal{H}^{k}(M); \alpha = \dd \epsilon \}\\
			\mathcal{H}^{k}(M) &= \mathcal{H}_{\mathsf{n}}^{k}(M)\oplus \{ \alpha \in\mathcal{H}^{k}(M); \alpha = \delta \gamma \}
		\end{split}
	\end{equation*}
	where $\mathcal{H}^{k}_{\mathsf{t}}(M)$ and $\mathcal{H}^{k}_{\mathsf{n}}(M)$ are the harmonic $k$--fields
	that are respectively tangential and normal.
\end{theorem}
Thus, every harmonic field is decomposed into either: (i) a tangential harmonic field plus an exact harmonic field
or (ii) a normal harmonic field plus a co-exact harmonic field.

Notice that by combining the first Friedrich decomposition in Theorem~\ref{thm:Friedrich_decomposition} 
with the Hodge decomposition in Theorem~\ref{thm:Hodge_boundary_general}, we
obtain an $L^{2}$~orthogonal decomposition of exact~$k$--forms as
\begin{equation*}
	\dd\Omega^{k-1}(M) = \dd\Omega^{k-1}_{\mathsf{n}}(M)\oplus \{ \alpha \in\mathcal{H}^{k}(M); \alpha = \dd \epsilon \}.
\end{equation*}
Indeed, for any $\epsilon\in\Omega^{k-1}(M)$ it cannot hold that $\dd\epsilon$ belongs to $\mathcal{H}^{k}_{\mathsf{t}}(M)$
or $\delta\Omega^{k+1}_{\mathsf{t}}(M)$ since that would imply that $\dd\epsilon$ also belongs to
$\{ \alpha \in\mathcal{H}^{k}(M); \alpha = \dd \epsilon \}$, which is disjoint to both.
Similarly, we get
\begin{equation*}
	\delta\Omega^{k+1}(M) = \delta\Omega^{k+1}_{\mathsf{t}}(M)\oplus \{ \alpha \in\mathcal{H}^{k}(M); \alpha = \delta \gamma \}.
\end{equation*}

Closely related to the harmonic $k$--fields $\mathcal{H}^k(M)$ are the \emph{harmonic $k$--forms},
given by $H^k(M) = \{ \alpha \in\Omega^k(M); \Delta \alpha = 0 \}$,
where $\Delta:=\delta\circ\dd + \dd\circ\delta$ is the Laplace--deRham operator.
For a closed manifold
it holds that $\mathcal{H}^k(M) = H^k(M)$.
However, in the presence of a boundary, $\mathcal{H}^k(M)$ is
strictly smaller than $H^k(M)$.

By combining the two versions of the Friedrich decomposition in Theorem~\ref{thm:Friedrich_decomposition}
we obtain four mutually orthogonal subspaces of~$\mathcal{H}^k(M)$. 
Altogether, we thus have six mutually orthogonal subspaces of~$\Omega^k(M)$,
which are given in the following table.

\begin{center}
	\begin{tabular}{ll}
		\toprule
		Short name & Definition \\
		\midrule
		$A^k_1(M)$ & $\dd \Omega^{k-1}_\mathsf{n}(M)$ \\
		$A^k_2(M)$ & $\delta \Omega^{k+1}_\mathsf{t}(M)$ \\
		$A^k_3(M)$ & $\mathcal{H}^k_\mathsf{n}(M)\cap\mathcal{H}^k_\mathsf{t}(M)$ \\
		$A^k_4(M)$ & $\mathcal{H}^k_\mathsf{n}(M)\cap \dd\Omega^{k-1}(M)$ \\
		$A^k_5(M)$ & $\mathcal{H}^k_\mathsf{t}(M)\cap \delta\Omega^{k+1}(M)$ \\
		$A^k_6(M)$ & $\dd\Omega^{k-1}(M)\cap \delta\Omega^{k+1}(M)$ \\
		\bottomrule
	\end{tabular}
\end{center}

Theorem~\ref{thm:Hodge_boundary_general} together with Theorem~\ref{thm:Friedrich_decomposition}
now yields the following result, which we naturally call the complete Hodge decomposition
for manifolds with boundary.

\begin{corollary}\label{cor:complete_hodge}
	$\Omega^k(M)$ admits the $L^2$~orthogonal
	decomposition
	\[
		\Omega^k(M) = \bigoplus_{l=1}^6 A^k_l(M) .
	\]
\end{corollary}

\begin{remark}\label{rmk:trivialA3}
	If $M$ does not have a boundary, then
	$A^k_4(M),A^k_5(M),A^k_6(M)$ are trivial, and Corollary~\ref{cor:complete_hodge}
	reduced to the ordinary Hodge decomposition for closed manifolds.
	In contrast, if~$M$ is the closure of a bounded open subset of~$\R^n$ with smooth boundary,
	then $A_3^k(M)$ is trivial, since a harmonic field on such a manifold 
	which is zero on the boundary, must be zero also in the interior.
	(The Laplace equation with Dirichlet boundary conditions is well-posed.)
\end{remark}

\subsection{Characterisation in terms of four fundamental subspaces} 
\label{sub:characterisation_in_terms_of_four_fundamental_subspaces}

The differential and co-differential induces four fundamental
subspaces of $\Omega^k(M)$ given by
\begin{equation*}
	\begin{split}
		\image\dd &= \dd\Omega^{k-1}(M), \\
		\ker\dd &= \{\alpha\in\Omega^k(M);\dd\alpha=0\}, \\
		\image\delta &= \delta\Omega^{k+1}(M), \\
		\ker\delta &= \{\alpha\in\Omega^k(M);\delta\alpha=0\}.
	\end{split}
\end{equation*}
Notice that $\image\dd \subset \ker\dd$ and $\image\delta\subset\ker\delta$.
Also notice that the intersection between any two of these
subspaces in general is non-empty.

Each of the mutually orthogonal spaces $A^k_l(M)$
can be characterised by orthogonal complements and intersections
of the four fundamental subspaces. 

\begin{proposition}\label{pro:fundamental_characterisation_of_hodge}
	It holds that
	\begin{equation*}
		\begin{split}
			A^k_1(M) &= (\ker\delta)^\bot \\
			A^k_2(M) &= (\ker\dd)^\bot \\
			A^k_3(M) &= (\image\dd)^\bot\cap(\image\delta)^\bot \\
			A^k_4(M) &= \ker\delta\cap\image\dd\cap(\image\delta)^\bot \\
			A^k_5(M) &= \ker\dd\cap\image\delta\cap(\image\dd)^\bot \\
			A^k_6(M) &= \image\dd\cap\image\delta \\
		\end{split}
	\end{equation*}
\end{proposition}

\begin{proof}
	From mutual orthogonality between $A^k_l(M)$ we obtain
	\begin{equation*}
		\begin{split}
			\image \dd &= A^k_1(M)\oplus A^k_4(M)\oplus A^k_6(M) \\
			\image \delta &= A^k_2(M)\oplus A^k_5(M)\oplus A^k_6(M) \\
			\ker \dd &= A^k_1(M)\oplus A^k_3(M)\oplus A^k_4(M)\oplus A^k_5(M)\oplus A^k_6(M) \\
			\ker \delta &= A^k_2(M)\oplus A^k_3(M)\oplus A^k_4(M)\oplus A^k_5(M)\oplus A^k_6(M) .
		\end{split}
	\end{equation*}
	Using Corollary~\ref{cor:complete_hodge},
	the result now follows from basic set operations.
\end{proof}


\begin{remark}
	A special case of Proposition~\ref{pro:fundamental_characterisation_of_hodge}
	is given in~\cite{CaDeGl2002}. Indeed, that paper
	gives a characterisation of the Helmholtz decomposition
	of vector fields on bounded domains of $\R^3$ in terms of
	the kernel and image of the $\grad$ and $\curl$ operators.
\end{remark}


\subsection{Special case of 2--manifolds} 
\label{sub:vector_calculus_on_2_manifolds}

In this section we analyse in detail the complete Hodge decomposition
in the case of 2--manifolds. The de~Rham complex and co-complex for a Riemannian 
2--manifold $(M,\mathsf{g})$ is
\[
	\xymatrix{ 
		\Omega^{0}(M)\ar[d]_{\star} \ar[r]^{\dd} & \Omega^{1}(M) 
		\ar[d]_{\star}\ar[r]^{\dd} & \Omega^{2}(M)\ar[d]_{\star} \\
		\Omega^{2}(M) \ar[r]^{\delta} & \Omega^{1}(M) \ar[r]^{\delta} & \Omega^{0}(M)
	}
\]
so the Hodge star maps $\Omega^1(M)$ isomorphically to itself. If $\alpha\in\Omega^1(M)$,
then $\star\star\alpha = -\alpha$, so the Hodge star induces an almost complex structure 
on~$M$. Also, since the Hodge star maps normal forms to tangential
forms, closed forms to co-closed forms, and exact forms to co-exact forms (and vice-versa),
it holds that $\star A^1_1(M) = A^1_2(M)$ and $\star A^1_4(M) = A^1_5(M)$.
It also holds that $\star A^1_3(M) = A^1_3(M)$ and $\star A^1_6(M) = A^1_6(M)$.

We now work out the complete Hodge decomposition of $\Omega^1(M)$ 
in some standard examples. These relevant later in the paper,
when we discuss the Lie algebra of conformal vector fields.

\begin{example}[Disk]\label{ex:disk1}
	Let $\disk$ be the unit disk in $\R^2$, equipped with the 
	Euclidean metric. Since the first co-homology group
	of $\disk$ is trivial, it holds that every closed 1--form is
	exact, and every co-closed 1--form is co-exact. Thus,
	$A_3^1(\disk)$, $A_4^1(\disk)$ and $A_5^1(\disk)$ are trivial, so
	all harmonic 1--fields are in $A_6^1(\disk)$. Thus, the complete
	Hodge decomposition for $\Omega^1(\disk)$ is
	\begin{equation}\label{eq:complete_hodgr_disk}
		\Omega^1(\disk) = A^1_1(\disk)\oplus A^1_2(\disk)\oplus A^1_6(\disk),
	\end{equation}
	where all the components are infinite dimensional.
\end{example}

\begin{example}[Annulus] \label{ex:annulus1}
	Let $\mathbb{A}$ be a standard annulus in $\R^2$, equipped with the 
	Euclidean metric. As for the disk, it holds that
	$A^1_3(\mathbb{A})$ is trivial (see Remark~\ref{rmk:trivialA3}). However, $A^1_4(\mathbb{A})$
	and $A^1_5(\mathbb{A})$ are not trivial. 
	Indeed, if $\alpha\in A^1_4(\mathbb{A})$, then
	$\alpha = \dd f$ where the function~$f$ fulfils
	$\delta\dd f = \Delta f = 0$. Also, since $\alpha$ is normal,
	it holds that $0 = i^*( \dd f) = \dd i^*(f)$, so $f$ is constant
	on each of its two connected components of the boundary.
	By uniqueness of the Laplace equation on $\mathbb{A}$
	it then holds that~$f$ is determined uniquely by the two boundary constants.
	Since the differential of a constant is zero, it is no restriction
	to assume that the constant at the outer boundary is zero,
	so we have that 
	$A^1_4(\mathbb{A}) = \{\dd f; \Delta f = 0, f|_{\pd_1 \mathbb{A}}=0, f|_{\pd_2\mathbb{A}}=\text{const} \}$.
	This space has dimension~1, and since $A^1_5(\mathbb{A}) = \star A^1_4(\mathbb{A})$ it also
	holds that $A^1_5(\mathbb{A})$ has dimension~1. Explicitly, in Cartesian coordinates, we have
	$A^1_4(\mathbb{A}) = \mathrm{span}\{ \alpha \}$, where $\alpha = \dd \ln(x^2+y^2)$,
	and $A^1_5(\mathbb{A}) = \mathrm{span}\{ \star\alpha \} = \mathrm{span}\{ (x\ud y - y \ud x)/(x^2+y^2) \}$.
	%
	%
	%
	Thus, the complete
	Hodge decomposition is
	\begin{equation}\label{eq:complete_hodgr_annulus}
		\Omega^1(\mathbb{A}) = A^1_1(\mathbb{A})\oplus A^1_2(\mathbb{A})
		\oplus A^1_4(\mathbb{A})\oplus A^1_5(\mathbb{A})\oplus A^1_6(\mathbb{A}),
	\end{equation}
	where $A_4^1(\mathbb{A})$ and $A_5^1(\mathbb{A})$ are 1--dimensional,
	and all the other components are infinite dimensional.
\end{example}

\newcommand{\T}{\mathbb{T}}

\begin{example}[Torus]
	Let $\T^2$ be the torus, equipped with the 
	metric inherited from~$\R^3$. Since $\T^2$
	is a closed manifold, $A_4^1(\T^2)$, $A_5^1(\T^2)$
	and $A_6^1(\T^2)$ are trivial.
	However, the first co-homology group
	is two dimensional, so $A_3^1(\T^2)$ is also
	two dimensional. Thus, the complete
	Hodge decomposition for $\Omega^1(\T^2)$ is
	\begin{equation}\label{eq:complete_hodgr_torus}
		\Omega^1(\T^2) = A^1_1(\T^2)\oplus A^1_2(\T^2)\oplus A^1_3(\T^2).
	\end{equation}
\end{example}

\begin{example}[Sphere]
	Let $S^2$ be the sphere, equipped with the 
	metric inherited from~$\R^3$. Since $S^2$
	is a closed manifold, $A_4^1(S^2)$, $A_5^1(S^2)$
	and $A_6^1(S^2)$ are trivial (see Remark~\ref{rmk:trivialA3}).
	Since the first co-homology group
	of $S^2$ is trivial, it holds that 
	$A_3^1(S^2)$ is trivial.
	Thus, the complete
	Hodge decomposition for $\Omega^1(S^2)$ is
	\begin{equation}\label{eq:complete_hodgr_sphere}
		\Omega^1(S^2) = A^1_1(S^2)\oplus A^1_2(S^2).
	\end{equation}
\end{example}

\section{Orthogonal decomposition of vector fields} 
\label{sec:orthogonal_decomposition_of_vector_fields}

The $L^{2}$~inner product on~$\Xcal(M)$ is given by
\begin{equation*}
	\pair{\xi,\eta}_{M} := \int_{M} \mathsf{g}(\xi,\eta)\vol
\end{equation*}
where $\vol$ is the volume form induced by~$\mathsf{g}$.
(Notice that we use the same notation as for the $L^{2}$~inner product on forms.)

In this section we show how the Hodge decomposition can be
used to obtain the $L^{2}$~orthogonal complement of the
Lie subalgebras of vector field discussed in Section~\ref{sec:lie_algebras_of_vector_fields} above. 
The approach is to find an isometry $\Xcal(M)\to\Omega^{1}(M)$
which maps the vector field subalgebra under study
onto one of the components in the Hodge decomposition.

For the first case of volume preserving vector fields is well known
that $\Xcal(M)$ can be orthogonally decomposed into divergence free
plus gradient vector fields. Using contraction with the metric
as an isometric isomorphism $\Xcal(M)\to\Omega^{1}(M)$,
this decomposition is expressed by the Hodge decomposition 
of~$1$--forms. (Equivalently, one may use contraction with the volume form instead,
which corresponds to the Hodge decomposition of ($n\!\!-\!\! 1$)--forms.)
The decomposition is essential in the derivation of the Euler equations for 
the motion of an ideal incompressible fluid. In this
case, the Lagrange multipliers in the projection has
a physical interpretation as the pressure in the fluid.


For the second case, it is well known that symplectic vector fields
can be identified with closed 1--forms, by contraction with the symplectic form.
Usually, flow equations
for symplectic vector fields are written in terms of the
Hamiltonian function, i.e., $\Xcal_{\omega}(M)$
is identified with~$\Fcal(M)$, 
and the equations are
expressed on $\Fcal(M)$ (an example is the equation for 
quasigeostrophic motion). This approach, corresponding to 
vorticity formulation in the case of the Euler fluid,
is viable in the setting of Hamiltonian
vector fields, i.e., those which have 
a globally defined Hamiltonian. However, on manifolds
which are not simply connected
(so that not every closed 1--form is exact), 
this
approach may not be feasible.
However, representation on the full space of vector fields,
using Lagrangian multipliers for orthogonal projection,
can always be used.

The third case of conformal vector fields is a new example.
For a manifold of dimension larger than 2, it follows from
a theorem by Liouville that the space of conformal vector fields
is finite dimensional. 
However, in the case of 2--manifolds, the space
of conformal vector fields can be (but does not have to be) 
infinite dimensional. The approach we follow
in this paper works in the case of flat 2--manifolds.
For closed manifolds,
this includes essentially only the 2--torus, for which
the set of conformal vector fields is finite dimensional.
However, in the case of a bounded domain of~$\R^{2}$,
the conformal vector fields correspond to all holomorphic
functions on this domain, and is thus infinite dimensional. The application
we have in mind is the derivation of Euler-Lagrange 
equations for variational problems on the space of conformal vector fields.
An example is geodesic motion on the infinite dimensional manifold
of conformal embeddings.

Each of the vector field subalgebras discussed in Section~\ref{sec:lie_algebras_of_vector_fields}
can be identified with one or several components in the complete Hodge decomposition.
However, the choice of isomorphism $\Xcal(M)\to\Omega^{1}(M)$
is different between the three basic cases of divergence free, symplectic, and
conformal vector fields. Once the isomorphism has been specified,
we use the short-hand notation $\xi\mapsto\xi^\flat$ 
for the map $\Xcal(M)\to\Omega^1(M)$
and $\alpha\mapsto\alpha^\sharp$
for its inverse.
Table~\ref{tab:decompositions} contains
an overview of the decompositions. Detailed expositions for each
case are given in the remaining part of this section.


\begin{table}[tp]%
	\centering%
	\addtolength{\tabcolsep}{0.01ex}
	\renewcommand{\arraystretch}{1.0}
	\smaller
	\begin{tabular}{ccl}
		\toprule 
		\textbf{Setting} & \textbf{Isomorphism} & \multicolumn{1}{l}{\textbf{Decompositions}} 
		\\ \midrule
			$M^{n},\mathsf{g}$ & 
			$\xi\mapsto\interior_\xi\mathsf{g}$ & 
			$
				\Xcal(M) = \!\!\!\underbrace{\Xcal_{\vol}(M)}_{
					\displaystyle
					\bigoplus_{l\neq 1} A^1_l(M)^\sharp
				}\!\!\!
				\oplus \, \underbrace{\grad(\Fcal_{0}(M))}_{
					\displaystyle
					A^1_1(M)^\sharp
				}
				%
			$
		\\ \cmidrule{3-3}
			& & 
			$\Xcal(M) = \!\!\!\!\! \underbrace{\Xcal_{\vol,\mathsf{t}}(M)}_{
					\displaystyle
					\bigoplus_{l \in \{2,3,5\}} \!\!\!\! A^1_l(M)^\sharp
				}
				\oplus\; \underbrace{\grad(\Fcal(M))}_{
					\displaystyle
					\bigoplus_{l \in \{1,4,6\}} \!\!\!\! A^1_l(M)^\sharp
				}
			$
		\\ \cmidrule{3-3}
			& & 
			$\Xcal(M) = \!\!\!\!\! \underbrace{\Xcal_{\vol}^{\mathrm{ex}}(M)}_{
				\displaystyle
				\bigoplus_{l \in \{2,5,6\}} \!\!\! A^1_l(M)^\sharp
			}
				\!\!\!\!\oplus\; \underbrace{\grad(\Fcal_{0}(M))}_{\displaystyle A^1_1(M)^\sharp}
				\oplus
				\underbrace{\mathcal{H}^1_\mathsf{n}(M)^\sharp}_{
					\displaystyle
					\bigoplus_{l \in \{3,4\}} \!\!\! A^1_l(M)^\sharp
				}
			$
		\\ \cmidrule{3-3}
			& & 
			$\Xcal(M) = \!\!\!\!\! \underbrace{\Xcal_{\vol,\mathsf{t}}^{\mathrm{ex}}(M)}_{
				\displaystyle
				\bigoplus_{l \in \{2,5\}} \!\! A^1_l(M)^\sharp
			}
				\!\oplus\; \underbrace{\grad(\Fcal(M))}_{\displaystyle
					\bigoplus_{l \in \{1,4,6\}} \!\!\! A^1_l(M)^\sharp
				}
				\oplus\;\;
				\underbrace{\mathcal{H}^1_0(M)^\sharp}_{
					\displaystyle
					A^1_3(M)^\sharp
				}
			$
		\\ \midrule 
			$\displaystyle \underset{(\text{almost Kähler})}{M^{2n},\mathsf{g},\omega}$ &
			$\xi\mapsto \interior_\xi\omega$ &
			$
				\Xcal(M) = \!\!\!\underbrace{\Xcal_{\omega}(M)}_{
					\displaystyle
					\bigoplus_{l\neq 2} A^1_l(M)^\sharp
				}\!\!\!
				\oplus \, \underbrace{\delta\Omega_\mathsf{t}^2(M)^\sharp}_{
					\displaystyle
					A^1_2(M)^\sharp
				}
			$
			\\ \cmidrule{3-3}
				& & 
				$\Xcal(M) = \!\!\!\!\! \underbrace{\Xcal_{\omega,\mathsf{t}}(M)}_{
						\displaystyle
						\bigoplus_{l \in \{1,3,4\}} \!\!\!\! A^1_l(M)^\sharp
					}
					\oplus\; \underbrace{\delta\Omega^{2}(M)^\sharp}_{
						\displaystyle
						\bigoplus_{l \in \{2,5,6\}} \!\!\!\! A^1_l(M)^\sharp
					}
				$
			\\ \cmidrule{3-3}
				& & 
				$\Xcal(M) = \!\!\!\!\! \underbrace{\Xcal_{\vol}^{\mathrm{ex}}(M)}_{
					\displaystyle
					\bigoplus_{l \in \{1,4,6\}} \!\!\! A^1_l(M)^\sharp
				}
					\!\!\!\!\oplus\; \underbrace{\delta\Omega_\mathsf{t}^2(M)^\sharp}_{
						\displaystyle A^1_2(M)^\sharp
					}
					\oplus
					\underbrace{\mathcal{H}^1_\mathsf{t}(M)^\sharp}_{
						\displaystyle
						\bigoplus_{l \in \{3,5\}} \!\!\! A^1_l(M)^\sharp
					}
				$
			\\ \cmidrule{3-3}
				& & 
				$\Xcal(M) = \!\!\!\!\! \underbrace{\Xcal_{\vol,\mathsf{t}}^{\mathrm{ex}}(M)}_{
					\displaystyle
					\bigoplus_{l \in \{1,4\}} \!\! A^1_l(M)^\sharp
				}
					\!\oplus\; \underbrace{\delta\Omega^{2}(M)^\sharp}_{\displaystyle
						\bigoplus_{l \in \{2,5,6\}} \!\!\! A^1_l(M)^\sharp
					}
					\oplus\;\;
					\underbrace{\mathcal{H}^1_0(M)^\sharp}_{
						\displaystyle
						A^1_3(M)^\sharp
					}
				$
		\\ \midrule 
			$\displaystyle \underset{(\text{flat})}{M^{2},\mathsf{g}}$ &
			$\xi\mapsto \interior_{\bar\xi}\mathsf{g}$ &
			$\Xcal(M) = \underbrace{\Xcalcon(M)}_{\displaystyle 
				\bigoplus_{l \neq \{1,2\}} \!\!\! A^1_l(M)^\sharp
			}
			\oplus \; \underbrace{\overline{\grad}(\Fcal_{0}(M))}_{\displaystyle
				A^1_1(M)^\sharp
			}
			\oplus \; \underbrace{\overline{\sgrad}(\Fcal_{0}(M))}_{\displaystyle
				A^1_2(M)^\sharp
			}
			$
		\\ \bottomrule
	\end{tabular}
	\caption{Various $L^{2}$~orthogonal decompositions of vector fields based on
	the Hodge decomposition. 
	The first column specifies the required setting.
	$M$ is a manifold, possibly with boundary, and
	the upper index denotes its dimension. 
	The second column specifies the
	isomorphism used to identify vector fields with 1--forms.
	The third column gives various $L^{2}$~orthogonal decompositions of
	$\Xcal(M)$, and corresponding 
	components in the Hodge decomposition (via the contraction map).
	The first component in each decomposition is
	a Lie subalgebra of vector fields.} \label{tab:decompositions}
\end{table}

\begin{remark}
	The requirement that the isomorphism $\Xcal(M)\to\Omega^{1}(M)$
	is an isometry
	can be weakened. Indeed, our basic requirement is that orthogonality
	is preserved, so it is enough that the map
	is conformal. However, in the examples in this paper
	the isomorphism will be an isometry.
\end{remark}

\subsection{Volume preserving vector fields on a Riemannian manifold} 
\label{sub:volume_preserving_case}

Let $(M,\mathsf{g})$ be a compact Riemannian manifold.
As isomorphism we use contraction with the metric, i.e., 
$\xi\mapsto\interior_\xi\mathsf{g}=:\xi^\flat$.

As before, let $\vol$ denote the volume form induces by the
Riemannian metric~$\mathsf{g}$.
Since $\xi$ is divergence free if and only
if $\interior_{\xi}\vol$ is closed, it follows from the formula
$\interior_{\xi}\vol = \star\xi^\flat$ that
$\xi$ is divergence free if and only if $\xi^\flat$
is co-closed, i.e., $\delta\xi^\flat = 0$.
Furthermore, $\xi\in\Xcal_{\vol}^\mathrm{ex}(M)$ if and only if
$\xi^\flat$ is co-exact, i.e., $\xi^\flat = \delta\alpha$
for some $\alpha\in\Omega^2(M)$. Thus, $\Xcalvol(M)^\flat = \ker \delta$
and $\Xcal_\vol^\mathrm{ex}(M)^\flat = \image \delta$,
which gives the first and the third decompositions in Table~\ref{tab:decompositions}.

For the case of tangential vector fields, we notice that
$\xi$ is a tangential vector field if and only if $\xi^\flat$ is
a tangential 1--form. Since elements in $A^1_4(M)$ and $A^1_6(M)$
are necessarily non-tangential, we obtain the second and the fourth
decompositions in Table~\ref{tab:decompositions}.

\subsection{Symplectic vector fields on an almost Kähler manifold} 
\label{sub:kahler_case}

Let $(M,\mathsf{g},\omega)$ be an almost Kähler manifold.
As isomorphism we use contraction with the symplectic form,
i.e., $\xi\mapsto\interior_\xi\omega=:\xi^\flat$. Due to the
almost Kähler structure, this isomorphism is isometric.

Since $\LieD_\xi\omega = \dd\interior_\xi\omega = \dd\xi^\flat$
it follows that $\xi$ is symplectic if and only if
$\xi^\flat$ is closed. Likewise, $\xi$ is Hamiltonian
if and only if $\xi^\flat$ is exact. Thus,
$\Xcal_\omega(M)^\flat = \ker\dd$ and $\Xcal_\mathrm{Ham}(M)^\flat = \image\dd$.
This yields the first and the third symplectic decompositions in
Table~\ref{tab:decompositions}.

For the case of tangential symplectic and Hamiltonian vector fields,
we notice that $\xi$ is a tangential vector field if and only if
$\xi^\flat$ is a normal 1--form. Since elements in $A^1_5(M)$ and $A^1_6(M)$
are necessarily non-tangential, we obtain the second and the fourth
symplectic decompositions in Table~\ref{tab:decompositions}.


\subsection{Conformal vector fields on a 2--manifold} 
\label{sub:new_example_conformal_case}

Let $(M,\mathsf{g})$ be a compact Riemannian 2--manifold, possibly with boundary.
If $(M,\mathsf{g})$ is flat, then there exists an orthogonal reflection map
$R\in\mathcal{T}_{1}^{1}(M)$, i.e.,
an orientation reversing isometry such that $R^{2}= \Id$. In that case, we may chose
coordinate charts such that
$\mathsf{g} = \dd x\otimes \dd x + \dd y\otimes \dd y$ and 
$R = \dd x\otimes \pd_{x} - \dd y \otimes \pd_{y}$.
If $\xi\in\Xcal(M)$ then we write $\bar\xi := R\xi$.
%
%
Since $R$ is an isometry, it holds that the isomorphism
\begin{equation*}
	\Xcal(M)\ni 
		\xi \mapsto \interior_{\bar\xi}\mathsf{g} =: \xi^\flat
	\in \Omega^{1}(M)
\end{equation*}
is isometric with respect to the $L^{2}$~inner 
products on $\Xcal(M)$ and $\Omega^{1}(M)$.
The following lemma is the key to obtaining the $L^{2}$~orthogonal
complement of~$\Xcalcon(M)$ in~$\Xcal(M)$.

\begin{lemma}\label{lem:con_to_harmonic}
	If $(M,\mathsf{g})$ is a flat 2--manifold
	then $\Xcalcon(M)^{\flat} = \mathcal{H}^{1}(M)$.
\end{lemma}

\begin{proof}
	It is enough to prove the assertion in local coordinates
	as above. Let $\xi = u\pd_x + v\pd_y \in \Xcalcon(M)$.
	Then $(u,v)$ must fulfil the Cauchy--Riemann equations.
	
	It holds that $\xi^\flat = u\ud x - v \ud y$.
	Thus, 
	\begin{equation*}
		\ud\xi^\flat = u_y \ud y\wedge\dd x - v_x \ud x\wedge\dd y = -(v_x + u_y)\ud x\wedge\dd y
	\end{equation*}
	and
	\begin{equation*}
		\delta\xi^\flat = \star \ud \star \xi^\flat = \star \ud (v\ud x + u\ud y)
		= \star (u_x - v_y)\ud x\wedge\dd y = u_x - v_y .
	\end{equation*}
	Hence, we see that $\dd\xi^\flat = 0$ and $\delta\xi^\flat = 0$ if and only
	if $(u,v)$ fulfils the Cauchy-Riemann equations. This proves the assertion.
\end{proof}

We now introduce ``reflected'' versions of the gradient and the skew gradient.
Indeed, for a function $F\in\Fcal(M)$ 
we define the \emph{reflection gradient} as $\overline{\grad}(F) = (\dd F)^\sharp$
and the \emph{reflection skew gradient} as $\overline{\sgrad}(F) = (\star\dd F)^\sharp$.
%
%
In local (flat) coordinates we have
\begin{equation*}
	\overline{\grad}(F) = F_{x}\pd_{x} - F_{y}\pd_{y},
	\qquad
	\overline{\sgrad}(F) = F_{y}\pd_{x} + F_{x}\pd_{y} .
\end{equation*}

%
%
%
%
%
%
%
%
%

Let $\Fcal_{0}(M) = \{ F\in\Fcal(M); F|_{\pd M} = 0 \}$, i.e., 
the functions that vanish on the boundary.
Using Lemma~\ref{lem:con_to_harmonic} 
together with the Hodge decomposition Theorem~\ref{thm:Hodge_boundary_general}
we obtain the following result.

\begin{theorem}\label{thm:conformal_decomposition}
	If $(M,\mathsf{g})$ is a flat 2--manifold, then
	the space of vector fields on $M$ admits the 
	$L^2$~orthogonal decomposition
	\[
		\Xcal(M) = \Xcal_\con(M)\oplus\overline{\grad}(\Fcal_0(M))\oplus
			\overline{\sgrad}(\Fcal_0(M))
	\]
	where each component is closed in $\Xcal(M)$ with respect to the
	Fréchet topology.
\end{theorem}

\begin{proof}
	First, we need to verify the following diagram:
	\[
	\xymatrix@1@R=5ex@C=0ex{
		\overline{\grad}(\Fcal_0(M))\ar[d]_{\flat} &\oplus&
		\overline{\sgrad}(\Fcal_0(M))\ar[d]_{\flat} &\oplus&
		\Xcalcon(M)\ar[d]_{\flat}
		\\
		\dd\Omega^0_\mathsf{n}(M) &\oplus&
		\delta\Omega^2_\mathsf{t}(M) &\oplus&
		\mathcal{H}^1(M)
	}
	\]
	It follows from the definition of normal forms that $\Omega^{0}_{\mathsf{n}}(M) = \Fcal_{0}(M)$.
	Thus, 
	\begin{equation*}
		\overline{\grad}(\Fcal_{0}(M))^{\flat} = \dd\Fcal_{0}(M) = \dd \Omega^{0}_{\mathsf{n}}(M).
		%
		%
	\end{equation*}
	Also, it holds that $\Omega^{0}_{\mathsf{n}}(M) = \star\Omega^{2}_{\mathsf{t}}(M)$, which gives
	\begin{equation*}
		\overline{\sgrad}(\Fcal_{0}(M))^{\flat} = 
		\star\dd\Omega^{0}_{\mathsf{n}}(M)
		= -\delta\Omega^{2}_{\mathsf{t}}(M)= \delta\Omega^{2}_{\mathsf{t}}(M) .
	\end{equation*}
	Next, it follows from Lemma~\ref{lem:con_to_harmonic} that $\Xcalcon(M)^{\flat} = \mathcal{H}^{1}(M)$.
	Now, from the Hodge decomposition of forms (Theorem~\ref{thm:Hodge_boundary_general})
	it follows that the subspaces of forms are $L^{2}$~orthogonal to each other.
	This implies that the subspaces
	of vector fields are also $L^{2}$~orthogonal to each other,
	since the isomorphism $\xi\to\xi^\flat$ is an isometry.
	Finally, since the isomorphism~$\flat:\Xcal(M)\to\Omega^{1}(M)$
	is a continuous (even smooth) vector space isomorphism with 
	respect to the Fréchet topologies on $\Xcal(M)$ and $\Omega^{1}(M)$,
	and since the subspaces of forms are topologically closed in~$\Omega^{1}(M)$,
	it follows that the vector field subspaces are
	topologically closed in~$\Xcal(M)$.
\end{proof}

\begin{example}[Disk]
	Let $M=\disk$ and let $(x,y)$ be Cartesian coordinates.
	If $\xi = u \pd_x + v \pd_y$, we take as 
	reflection map $R(\xi) = u\pd_x - v\pd_y$ (corresponding to
	complex conjugation of $z = x + \i y$). It follows from equation~\eqref{eq:complete_hodgr_disk}
	that $\mathcal{H}^1(\disk) = A^1_6(\disk)$.
	Thus, by Theorem~\ref{thm:conformal_decomposition} we get that
	the set of conformal vector fields on $\disk$ (corresponding to holomorphic functions on $\disk$)
	is isometrically isomorphic to the set of simultaneously exact and co-exact 1--forms.
\end{example}

\begin{example}[Annulus]
	Let $M=\mathbb{A}$ and let $(x,y)$ be Cartesian coordinates.
	We use the same reflection map as for the disk.
	It follows from equation~\eqref{eq:complete_hodgr_annulus}
	that $\mathcal{H}^1(\disk) = A^1_4(\disk)\oplus A^1_5(\disk) \oplus A^1_6(\disk)$.
	The special harmonic fields $\alpha$ and $\star\alpha$ (see Example~\ref{ex:annulus1}) corresponds to
	the holomorphic functions~$\i/z$ and~$1/z$.
	In complex analysis it is well known that these functions have a special role
	(e.g., in the calculus of residues).
	%
\end{example}

\begin{example}[Torus]
	Let $M=\mathbb{T}$ and let $(\theta,\phi)$ be the standard angle coordinates.
	(Notice that $\theta,\phi$ are not smooth functions on $\mathbb{T}$, but $\dd\theta$
	and $\dd\phi$ are well defined smooth 1--forms.)
	We use the reflection map $(\pd_\theta,\pd_\phi)\mapsto (\pd_\theta,-\pd_\phi)$.
	It follows from equation~\eqref{eq:complete_hodgr_torus} and Theorem~\ref{thm:conformal_decomposition}
	that the set of conformal vector fields is only two dimensional, generated by
	the pure translations $\pd_\theta$ and $\pd_\phi$.
	%
	%
\end{example}

\begin{counterexample}[Sphere]
	It follows from equation~\eqref{eq:complete_hodgr_sphere} that the space
	of harmonic fields on the sphere is trivial. However, the space
	of conformal vector fields of the sphere corresponds to the Möbius 
	transformations (by identifying the plane with the Riemann sphere)
	and is thus six dimensional, see e.g.~\cite{Sh1997}. Thus, it is \emph{not} possible to identify
	the space of conformal vector fields with the harmonic 1--fields in this
	case. The reason is that $S^2$ is not flat, so Theorem~\ref{thm:conformal_decomposition}
	does not apply.
\end{counterexample}

\section{Application to conformal variational problems} 
\label{sec:application_examples}

In this section we show how the orthogonal decomposition of conformal
vector fields can be used to derive differential equations for
variational problems involving conformal vector fields
on a bounded domain~$U\subset \R^{2}$ (or equivalently, variational
problems involving holomorphic functions on a bounded complex domain~$U$).
Throughout this section we identify complex valued functions on~$U$ with
vector fields on~$U$.

%

\subsection{Computing the conformal projection} 
\label{sub:computing_the_conformal_projection}

Given a vector field $\xi\in\Xcal(U)$, 
there is a direct way to compute the orthogonal
projection $\mathrm{Pr}_\con:\Xcal(U)\to\Xcalcon(U)$,
which does not involve solving a set of partial differential equations.
The approach is to use the \emph{Bergman kernel}~\cite{DuSc2004},
i.e., the reproducing kernel~$K_U(z,\cdot)$ of the \emph{Bergman space}~$A^2(U)$,
which is the Hilbert space obtained by completion of $\Xcalcon(U)$
with respect to the $L^2$~inner product. For any $f\in A^2(U)$
it then holds that $f(z) = \inner{f,K_U(z,\cdot)}_U$,
where $\inner{f,g}_U := \int_U f\bar g \ud A$ is the complex
$L^2$~inner product. The existence of $K_U(z,\cdot)$ follows from Riesz 
representation theorem since point-wise evaluation of functions in~$A^2(U)$ 
is continuous in the $L^2$~topology.
Now, for a general complex valued function $f\in L^2(U)$,
its orthogonal projection to $A^2(U)$ is given by $f_\con(z) = \inner{f,K_U(z,\cdot)}_U$.
Since $\xi\in\Xcal(U)$ implies that $\mathrm{Pr}_\con(\xi)\in\Xcalcon(U)$
and since $\Xcalcon(U)\subset A^2(U)$
it must hold that $\mathrm{Pr}_\con(\xi) = \inner{\xi,K_U(z,\cdot)}_U$.

For the case of the unit disk, it holds that $K_\disk(z,\zeta) = \frac{1}{\pi}\frac{1}{(1-\bar{z} \zeta)^2}$.
For other domains, the kernel function is given by $K_U(z,\zeta) =
K_\disk(\varphi(z),\varphi(\zeta))\varphi'(\zeta)\overline{\varphi'(z)}$,
where $\varphi$ is a conformal mapping $U\to \disk$.

A basis for $\Xcal(\disk)$ is given by $\{ z^m \bar z^n \}_{m,n\geq 0}$.
Thus, if $\xi\in\Xcal(\disk)$ is expanded in this basis, we may compute
$\mathrm{Pr}_\con(\xi)$ by applying $\mathrm{Pr}_\con$ to each of the basis
elements. Indeed, if $e_{mn}(z) = z^m\bar z^n$ then
\begin{equation*}
	\begin{split}
		\mathrm{Pr}_\con(e_{mn})(z) &= \int_\disk \frac{1}{\pi}\frac{1}{(1-z\bar\zeta)^2}\zeta^m\bar\zeta^n\ud A(\zeta)
		\\
		&= \int_\disk \frac{1}{\pi}\sum_{p=1}^{\infty} z^p\bar\zeta^p\zeta^m\bar\zeta^n \ud A(\zeta)
		= \left\{
		\begin{matrix}
			\frac{m-n+1}{m+1}z^{m-n} & m \geq n \phantom{.} \\
			0 & m < n.
		\end{matrix}
		\right.
	\end{split}
\end{equation*}

Another result that may be useful is the following.

\def\eqalign#1{\null\,\vcenter{\openup\jot \mathsurround=0pt \ialign{\strut
     \hfil$\displaystyle{##}$&$ \displaystyle{{}##}$\hfil \crcr#1\crcr}}\,}
\def\phi{\varphi}
\def\D{\mathbb D}
\def\C{\mathbb C}
\def\PrCon{\mathrm{Pr}_{\con}}
\def\X{\mathfrak X}
\def\dlangle{\langle\!\langle}
\def\drangle{\rangle\!\rangle}
\def\e{{\rm e}}
\def\i{{\rm i}}
\def\d{{\rm d}}
\def\R{{\mathbb R}}

\begin{proposition}\label{pro:projection_property}
	Let $\xi\in\X(U)$ and $\psi\in\Xcalcon(U)$. If $\psi(z)\ne 0$ for all $z\in U$ then $\PrCon(\xi)=0$ if
	and only if $\PrCon(\bar \psi \xi)=0$.
\end{proposition}

\begin{proof}
	We have
	\begin{align*}
		\PrCon(\xi)=0 & \iff \pair{ \xi,\eta}=0 & \forall \eta\in\Xcalcon(U) \\
		&\iff \pair{\bar \psi \xi,\eta/\psi} = 0 & \forall \eta\in\Xcalcon(U) \\
		&\iff \pair{\bar \psi \xi,\rho} = 0 & \forall \rho\in\Xcalcon(U) \\
		&\iff \PrCon(\bar \psi \xi) = 0. &
	\end{align*}
\end{proof}

%



\subsection{Integration by parts} 
\label{sub:integration_by_parts}


Standard variational calculus makes frequent use of integration by parts in order to
``isolate'' a virtual variation from derivatives. Usually, the boundary term
appearing either vanishes (in the case of a space of tangential vector fields),
or it can be treated separately giving rise to natural boundary conditions
(in the case of a space 
where vector fields can have arbitrary small compact support). 
However, in the case
of conformal vector fields, there is always a global dependence between
interior points, and points on the boundary (due to the Cauchy--Riemann equations).
Hence, in the conformal case, we need an appropriate analogue of
integration by parts which avoids boundary integrals.
Formally, we may proceed as follows. Let $\pd_{z}:\Xcalcon(U)\to\Xcalcon(U)$
be the complex derivative. Then we are looking for the adjoint of this
operator with respect to the $L^{2}$~inner product. That is, an
operator $\pd^{\trans}_{z}:\Xcalcon(U)\to\Xcalcon(U)$ such that
\begin{equation*}
	\pair{\xi,\pd_{z}\eta}_{U} = \pair{\pd^{\trans}_{z}\xi,\eta}_{U}, \qquad \forall\; \xi,\eta\in\Xcalcon(U).
\end{equation*}
%
%

Explicitly, it is most easily done on the disk.
Indeed, for $\xi, \eta\in\Xcalcon(\D)$, denoting $\pd_{z}\xi = \xi_{z}$, we have
\begin{equation*}
	 \pair{ \xi,\eta_z}_{\D} = \pair{(z^2\xi)_z,\eta}_\D,
\end{equation*}
as can be seen from the following.
Every element of $\Xcalcon(\D)$ has a convergent Taylor series, and the monomials $z^n$ form a basis for $\Xcalcon(\D)$ that is orthogonal with respect to both the real and the complex $L^2$~inner product, for
\begin{equation*}
\eqalign{
\inner{ z^m,z^n}_{\D} &= \int_{\D} \xi \bar \eta \ud A \cr
&=  \int_0^1 \left( r \ud r \int_0^{2\pi} \d\theta\, r^m \e^{\i m \theta} r^n \e^{-\i n\theta}\right) \cr
&= \frac{2\pi}{m+n+2}\delta_{m,n} \cr
}
\end{equation*}
and $\pair{\xi,\eta}_{\D} = \real\,\inner{ \xi,\eta}_{\D}$.
Therefore, expanding $\xi = \sum_{n=0}^\infty \xi_n z^n$ and $\eta=\sum_{m=0}^\infty \eta_m z^m$,
\begin{equation*}
\eqalign{
	\inner{ \xi,\eta_z}_\D &= \inner{\sum_{n=0}^\infty \xi_n z^n,(\sum_{m=0}^\infty \eta_m z^m)_z}_\D
	\cr &
	= \sum_{n,m=0}^\infty \xi_n \bar \eta_m \inner{ z^n, m z^{m-1}}_\D 
	\cr &
	= \sum_{n,m=0}^\infty \xi_n \bar \eta_m \frac{2\pi m}{n+m+1}\delta_{n,m-1} 
	\cr &
	= \sum_{n=0}^\infty \pi  \xi_n \bar \eta_{n+1}
}
\end{equation*}
while
\begin{equation*}
\eqalign{
\inner{ (z^2 \xi)_z, \eta}_\D &= \sum_{n,m=0}^\infty \xi_n \bar \eta_m \inner{ (n+2)z^{n+1},z^m}_\D \cr
&=  \sum_{n,m=0}^\infty \xi_n \bar \eta_m (n+2)\frac{2\pi}{n+m+3}\delta_{n+1,m} \cr
&= \sum_{n=0}^\infty \pi \xi_n \bar \eta_{n+1}.
}
\end{equation*}
Since the two complex inner products are equal, their real parts are equal, establishing the proposition.
On domains other than the unit disk, the formula for $\pd_{z}^{\trans}$ 
is not as simple. However, it can be computed if the domain
is the image of the unit disk under a known conformal embedding $\varphi:\disk\to U$.
(Due to the Riemann mapping theorem, such an embedding always exists
and is unique up composition from the left with the three dimensional submanifold of disk preserving 
Möbius transformations.)
For $\xi,\eta\in\Xcalcon(\varphi(\D))$, there are holomorphic functions 
\begin{equation*}
	 \chi_1 := (z^2 \varphi_z)_z\circ\varphi^{-1}, \quad \chi_2 := (z^2 \varphi_z^2)\circ\varphi^{-1}
\end{equation*}
on $\varphi(\D)$ such that
\begin{equation*}
	 \pair{ \xi, \eta_z}_{\varphi(\D)} = \pair{ |\varphi_z|^{-2}(\chi_1 \xi + \chi_2 \xi_z, \eta}.
\end{equation*}
Indeed, we have
\begin{equation*}
\eqalign{  \pair{ \xi, \eta_z}_{\varphi(\D)} &= 
\pair{ \varphi_z \xi\circ \varphi, \varphi_z \eta_z\circ\varphi}_\D \cr
&= \pair{ \varphi_z \xi\circ \varphi, (\eta\circ\varphi)_z}_\D \cr
&= \pair{ (z^2 \varphi_z \xi\circ\varphi)_z, \eta\circ\varphi}_\D \cr
&= \pair{(\varphi_z\circ\varphi^{-1})^{-1}(z^2 \varphi_z \xi\circ\varphi)_z)\circ \varphi^{-1}, (\varphi_z\circ\varphi^{-1})^{-1} \eta}_{\varphi(\D)} \cr
&= \pair{(|\varphi_z\circ\varphi^{-1}|^{-2}(z^2 \varphi_z \xi\circ\varphi)_z)\circ \varphi^{-1}, \eta}_{\varphi(\D)} \cr
&= \pair{|\varphi_z\circ\varphi^{-1}|^{-2}(\chi_1 \xi + \chi_2 \xi_z), \eta}_{\varphi(\D)}.
}
\end{equation*}
Thus, we have $\pd^{\trans}_{z} \xi = \PrCon\big( |\varphi_z\circ\varphi^{-1}|^{-2}(\chi_1 \xi + 
\chi_2 \xi_z) \big)$,
where $\PrCon:\Xcal(\varphi(\disk))\to\Xcalcon(\varphi(\disk))$ is the $L^{2}$~orthogonal
projection onto conformal vector fields.
Notice that $\pd^{\trans}_{z}$ depends on the domain~$U$, and is thus non-local.
The Hodge decomposition for conformal vector fields, developed in the previous section,
together with this formula now allow the calculation of the
equations of motion for any Lagrangian density on $\Xcalcon(\varphi(\D))$. 




\subsection{Example (Conformal stationary problem)} 
\label{sub:example_1_complex_time_flow}


Let $V\in\Fcal(\R^{2})$ and consider the Lagrangian 
density $\mathcal L(\xi,\xi_z) = \frac{1}{2}\abs{\xi_z}^2+V(\xi)$. 
Let $S:\Xcalcon(U)\to\R$ be the corresponding action $S(\xi) = \int_U \mathcal L(\xi(z),\xi_z(z))\ud A(z)$
and consider the variational problem:
\begin{equation*}
	\text{Find $\xi\in\Xcalcon(U)$ such that $\frac{\delta S}{\delta\xi}(\xi)\cdot\eta = 0$
	for all variations $\eta\in\Xcalcon(U)$.}
\end{equation*}
Direct calculations yield
\begin{equation*}
\eqalign{
\frac{\delta S}{\delta\xi}(\xi)\cdot\eta &= 
\pair{\xi_z,\eta_z}_{U}+ \pair{ \grad(V)\circ\xi,\eta}_{U} \cr
&= \pair{\pd^{\trans}_{z}\xi_z,\eta}_{U} + \pair{ \grad(V)\circ\xi,\eta}_{U} \cr
&= \pair{\pd^{\trans}_{z}\xi_z+\grad(V)\circ\xi,\eta}_{U}.
}
\end{equation*}
We require this to vanish for all $\eta\in\Xcalcon(\D)$. That is, the 
first term in the inner product must be orthogonal to all conformal vector fields, i.e.,
\begin{equation*}
	 \PrCon\left(\pd^{\trans}_{z}\xi_z+\grad(V)\circ\xi\right) = 0.
\end{equation*}
Since $\pd^{\trans}_{z}\xi_z$ is already holomorphic we get
\begin{equation*}
	 \pd^{\trans}_{z}\xi_z+\PrCon\big(\grad(V)\circ\xi\big) = 0.
\end{equation*}
Now, using the orthogonal decomposition of conformal vector fields, derived 
in Section~\ref{sub:conformal_vector_fields} above, we introduce 
Lagrange multipliers $F,G\in\Fcal_{0}(U)$ for the constraints, giving the
differential equation
\begin{equation*}
	\begin{split}
		\pd^{\trans}_{z}\xi_z+\grad(V)\circ\varphi &= 
			\overline{\grad}(F) + \overline{\sgrad}(G)
		\\
		\frac{\pd\xi}{\pd\bar z} &= 0 
		\\
		F|_{\pd U} &= G|_{\pd U} = 0
	\end{split}
\end{equation*}
where $\frac{\pd\xi}{\pd\bar z} = 0$ is short way to write the Cauchy-Riemann equations.
Notice that it is not certain that the original
variational problem is well-posed. That requires
additional assumptions on the function~$V$.

\subsection{Example (Conformal wave equation)} 
\label{sub:example_2}


Adding time to the previous example, and denoting $\frac{\dd }{\dd t}\xi = \xi_{t}$,
we consider the Lagrangian density
corresponding to a nonlinear conformal wave equation
\begin{equation*}
	\mathcal{L}(\xi,\xi_z,\xi_t) = \frac{1}{2}\abs{\xi_t}^2 - 
		\frac{1}{2}\abs{\xi_z}^2-V(\xi) .
\end{equation*}
Requiring that the action $S(\xi) = \int_{0}^{1}\int_U \mathcal L(\xi(z,t),\xi_z(z,t))\ud A(z) \ud t$ 
be stationary on paths in $\Xcalcon(U)$ fixed at the initial and final times yields
\begin{equation*}
	 \pair{-\xi_{tt} - \pd^{\trans}_{z}\xi_z-\grad(V)\circ\xi,\eta}_{U} = 0
\end{equation*}
and so the equations of motion are
\begin{equation*} 
	 \xi_{tt} + \pd^{\trans}_{z}\xi_z + \PrCon\big(\grad(V)\circ\xi\big) = 0,
\end{equation*}
or, spelled out explicitly using Lagrange multipliers, we get
the differential equation
\begin{equation}\label{eq:conformal_wave_eq_2}
	\begin{split}
		 \xi_{tt} + \pd^{\trans}_{z}\xi_z + \grad(V)\circ\xi &= 
		\overline{\grad}(F) + \overline{\sgrad}(G),
		\\
		\frac{\pd \xi}{\pd \bar z} &= 0 \\ 
		F|_{\pd U} &= G|_{\pd U} = 0 .
	\end{split}
\end{equation}


Consider now the case $U=\D$ and 
$V(z) = c|z|^2/2$ for a constant $c\in\R$. This gives the 
partial differential equation

\begin{equation*}
 \xi_{tt} + (z^2 \xi_z)_z + c \xi = 0 .
\end{equation*}
Expanding~$\xi$ in the monomial 
basis~$\xi = \sum_{m=0}^\infty \xi_m z^m$, we get
\begin{equation*}
	 (\xi_m)_{tt} + (m^2+m+c) \xi_m = 0,
\end{equation*}
i.e., a set of uncoupled harmonic oscillators. In particular,
there are an infinite number of first integrals, given by
\begin{equation*}
	I_{m}(\xi,\dot\xi) = \frac{1}{2}\abs{\dot\xi_{m}}^{2} +
		\frac{1}{2}(m^{2}+m+c)\abs{\xi_{m}}^{2}.
\end{equation*}



\subsection{Example (Geodesic conformal flow equation)} 
\label{sub:example_3}

Let $\Emb(\D,\R^2)$ denote the set of embeddings $\D\to\R^2$.
This set has the structure of a Fréchet-Lie manifold.
Although this manifold is not a group, it has many
similarities with the diffeomorphism group $\Diff(\D)$. 
First of all, we notice that $\Diff(\D)$ is a submanifold of $\Emb(\D,\R^2)$.
Secondly, it holds that the tangent space at the identity is equal
to the set of all vector fields on $\D$, i.e., $T_\Id\Emb(\D,\R^2) = \Xcal(\D)$,
which, as reviewed earlier, carries the structure of a Fréchet-Lie algebra
with the vector field commutator.
(Recall that the subalgebra $\Xcal_\mathsf{t}(\D)$ of tangential vector fields
is the tangent space at the identity of~$\Diff(\D)$.)

A Riemannian metric on $\Emb(\D,\R^2)$ is given by
\begin{equation}\label{eq:metric_on_emb}
	T_\varphi\Emb(\D,\R^2)\times T_\varphi\Emb(\D,\R^2)
	\ni (u,v) \longmapsto \pair{u\circ\varphi^{-1},v\circ\varphi^{-1}}_{\varphi(\D)} \in \R .
\end{equation}
Notice that this metric is invariant under the group $\Diff(\D)$ acting on $T\Emb(\D,\R^2)$ by
composition from the right.

Our aim is to derive the geodesic equation with respect to the metric~\eqref{eq:metric_on_emb}
restricted to the submanifold of conformal embeddings
\[
\Con(\D,\R^2) = \{ \varphi \in \Emb(\D,\R^2); \varphi^*\mathsf{g} = F\mathsf{g} \},
\]
where $\mathsf{g}$ is the Euclidean metric on $\R^2$.
By right translation, the tangent space at $\varphi\in\Con(\D,\R^2)$
can be identified with a conformal vector field over the domain $\varphi(\D)$.
Indeed, we have the isomorphism 
\begin{equation*}
	T\Con(\D,\R^2)\ni (\varphi,\dot\varphi) \mapsto (\varphi,\underbrace{\dot\varphi\circ\varphi^{-1}}_{\xi})
	\in\Con(\D,\R^2)\times \Xcalcon(\varphi(\D)),
\end{equation*}
where~$\Con(\D,\R^2)\times \Xcalcon(\varphi(\D))$ should be thought of as a vector bundle
over $\Con(\D,\R^2)$.

\newcommand{\eps}{\varepsilon}

In the language of Lagrangian mechanics, we have the Lagrangian
\begin{equation*}
	L(\varphi,\dot\varphi) = \frac{1}{2}\pair{ \dot\phi\circ\phi^{-1},\dot\phi\circ\phi^{-1}}_{\phi(\D)}.
\end{equation*}
We would like to derive the Euler-Lagrange equations, but using the variables $(\varphi,\xi)$
instead of $(\varphi,\dot\varphi)$. In doing so, we first notice that if $\varphi_\eps$
is a variation of a curve $\varphi(t)$ and $\xi_\eps = \dot\varphi_\eps\circ\varphi^{-1}_\eps$, then
\begin{equation*}
	\frac{\dd}{\dd\eps}\Big|_{\eps=0} \xi_\eps = \dot\eta + \LieD_\eta\xi ,
\end{equation*}
where $\varphi_\eps = \exp(\eps\eta)\circ\varphi$ with $\eta\in\Xcalcon(\varphi(\D))$,
see~\cite{GaMaRa2010preprint,MoPeMaMc2011_preprint}.
In addition, it holds that
\begin{equation*}
	\frac{\ud}{\ud \eps}\Big|_{\eps=0}\frac{1}{2}\pair{\xi,\xi}_{\varphi_\eps(\D)}
	= \pair{\LieD_\eta\xi+ \divv(\eta)\xi,\xi}_{\varphi(\D)} .
\end{equation*}
This equality follows by straightforward 
calculations and the fact that $\LieD_\xi\mathsf{g} = \divv(\xi)\mathsf{g}$
for any $\xi\in\Xcalcon(\varphi(\D))$, as derived in Section~\ref{sub:conformal_vector_fields}
above. 

Using these relations, the variational principle now yields
\begin{equation*}
	\begin{split}
		0 &= \frac{\dd}{\dd\eps}\Big|_{\eps=0} \int_0^1 L(\varphi_\eps,\dot\varphi_\eps) \ud t
			= \frac{\dd}{\dd\eps}\Big|_{\eps=0} \int_0^1 \frac{1}{2}\pair{\xi_\eps,\xi_\eps}_{\varphi_\eps(\D)}\ud t
		\\
		&= \int_0^1 \pair{\dot\eta + 2\LieD_\eta\xi  + \divv(\eta)\xi,\xi}_{\varphi(\D)}
		\\
		&=
		\int_{\varphi(\D)} \Big( 
			\mathsf{g}(\dot\eta,\xi)\vol + 
			\underbrace{\LieD_{\eta}}_{\ud\interior_{\eta}}\big(
				\mathsf{g}(\xi,\xi)\vol
			\big)
			- \mathsf{g}(\xi,\xi)\underbrace{\divv(\eta)\vol}_{\ud\interior_{\eta}\vol}
		\Big)
		\\
		&=
		\int_{\varphi(\D)} \Big( 
			\mathsf{g}(\dot\eta,\xi)\vol +
			\ud\big(
				\mathsf{g}(\xi,\xi)\interior_{\eta}\vol
			\big)
			-
			\ud\big(
				\mathsf{g}(\xi,\xi)\interior_{\eta}\vol
			\big)
			+
			\ud \mathsf{g}(\xi,\xi)\wedge\interior_{\eta}\vol
		\Big)
		\\
		&=
		\int_{\varphi(\D)} \Big( \mathsf{g}(\dot\eta,\xi)\vol
		+ \ud \mathsf{g}(\xi,\xi) \big)\wedge
		\interior_{\eta}\vol \Big)
		\\
		&= \pair{\dot\eta,\xi}_{\varphi(\D)} 
		+ \pair{\grad(\abs{\xi}^2),\eta}_{\varphi(\D)} .
	\end{split}
\end{equation*}
Next, since
\begin{equation*}
	\begin{split}
		\frac{\dd}{\dd t}\pair{\eta,\xi}_{\varphi(\D)} &= \pair{\dot\xi,\eta}_{\varphi(\D)} + \pair{\xi,\dot\eta}_{\varphi(\D)}
		+ \int_{\varphi(\D)}\LieD_\xi (\mathsf{g}(\xi,\eta)\vol) \\
		&= \pair{\dot\xi,\eta}_{\varphi(\D)} + \pair{\xi,\dot\eta}_{\varphi(\D)} + 
			\pair{\xi,\LieD_\xi\eta + 2\divv(\xi)\eta}_{\varphi(\D)}
	\end{split}
\end{equation*}
and since the variation $\eta$ vanish at the endpoints, we get
\begin{equation*}
	0 = \pair{\dot\xi + 2\divv(\xi)\xi - \grad(\abs{\xi}^2),\eta}_{\varphi(\D)}
	+ \pair{\xi,\LieD_\xi\eta}_{\varphi(\D)} .
\end{equation*}
Using now that $\LieD_\xi\eta = \xi'\eta - \eta'\xi = 2\xi'\eta - (\eta\xi)'$ we get
\begin{equation*}
	0 = \pair{\dot\xi + 2\divv(\xi)\xi - \grad(\abs{\xi}^2) + 2\xi\overline{\xi'} 
	- \overline{\xi}\pd_z^\trans\xi,\eta}_{\varphi(\D)}
\end{equation*}
Finally, from the relations $\divv(\xi) = 2\real(\xi')$ and
$\grad(\abs{\xi}^2) = 2\xi\overline{\xi'}$, and the
decomposition in Theorem~\ref{thm:conformal_decomposition},
we obtain the strong geodesic equation
\begin{equation}\label{eq:conforma_geodesic_equation_strong}
	\begin{split}
		\dot\xi + 2\divv(\xi)\xi - \overline{\xi}\pd_z^\trans\xi
		&= \overline{\grad}(F) + \overline{\sgrad}(G) \\
		\dot\varphi &= \xi\circ\varphi \\
		\frac{\pd\xi}{\pd\bar z} &= 0 \\
		F|_{\pd\varphi(\D)} &= G|_{\pd\varphi(\D)} = 0 .
	\end{split}
\end{equation}
Notice that the first equation contains the operator
$\pd_z^\trans$, which depends on the domain~$\varphi(\D)$.
Thus, the first equation for $\dot\xi$ depends on
the second equation for $\dot\varphi$. This is different from
``usual'' Euler equations, where the equation for the reduced variable~$\xi$
is independent of~$\varphi$. From a geometric mechanics point of view (cf.~\cite{MaRa1999}),
the reason for this coupling is that the symmetry group of the Lagrangian 
is smaller than the configuration space.

For further information of this geodesic equation, its application
in image registration, and a derivation using the more general class
of $H^1_\alpha$~metrics, see~\cite{MaMcMoPe2011,MoPeMaMc2011_preprint}.

\bibliographystyle{model1-num-names}
\bibliography{../Papers/References}


\end{document}